%% file: main.tex
\documentclass{article}
\usepackage{amsfonts,amsmath, amssymb,latexsym,comment, enumerate,mathrsfs}
\usepackage[OT2,OT1]{fontenc}

\usepackage{multirow}
\def\Z{{\mathbb Z}}
\def\A{{\mathbb A}}

\def\min{{\rm min}}

\def\SO{{\rm SO}}
\def\PSO{{\rm PSO}}

\def\B{{\mathcal B}}

\def\irr{{\rm irr}}

\def\Span{{\rm Span}}

\def\Vol{{\rm Vol}}
\def\R{{\mathbb R}}
\def\F{{\mathbb F}}

\renewcommand{\L}{\mathcal{L}}

\def\FF{{\mathcal F}}

\def\GG{{\mathcal G}}
\def\RR{{\mathcal R}}
\def\Q{{\mathbb Q}}

\def\Z{{\mathbb Z}}

\def\F{{\mathbb F}}
\def\Q{{\mathbb Q}}
\def\C{{\mathbb C}}

\def\W{{\mathcal W}}

\def\fz1{{F_{\Z,1}}}

\def\SO{{\rm SO}}
\def\PSO{{\rm PSO}}
\def\max{{\rm max}}

\def\dist{{\rm dist}}

\newcommand{\kM}{\mathfrak M}
\newcommand{\km}{\mathfrak m}
\newcommand{\kS}{\mathfrak S}

\newtheorem{theorem}{Theorem}[section]

\newtheorem{lemma}[theorem]{Lemma}

\newtheorem{proposition}[theorem]{Proposition}

\newenvironment{proof}{\noindent {\bf Proof:}}{$\Box$ \vspace{2 ex}}

\usepackage{xcolor,color,url,lmodern}

\title{On the squarefree values of $a^4+b^3$}

\author{Gian Cordana Sanjaya and Xiaoheng Wang}

\begin{document}
\maketitle

\begin{abstract}
    In this article, we prove that the density of integers $a, b$ such that $a^4+b^3$ is squarefree, when ordered by $\max\{|a|^{1/3},|b|^{1/4}\}$, equals the conjectured product of the local densities. 
    
    We show that the same is true for polynomials of the form $\beta a^4 +  \alpha b^3$ for any fixed integers $\alpha$ and $\beta$. We give an exact count for the number of pairs $(a,b)$ of integers with $\max\{|a|^{1/3},|b|^{1/4}\}<X$ such that $\beta a^4 +  \alpha b^3$ is squarefree, with a power-saving error term.
\end{abstract} 

\input{intro2}

\input{embed2}

\input{prelim_count2}
\input{circle_count}
\input{proof}

\subsection*{Acknowledgments}
It is a pleasure to thank Manjul Bhargava for many helpful comments. The first named author is supported by the University of Waterloo through an MURA project. The second named author is
supported by an NSERC Discovery Grant.

\end{document}

%% file: intro2.tex
\section{Introduction}

A classical question in analytic number theory is to determine the probability that a given polynomial $F$ with integer coefficients takes squarefree values when evaluated at random integers. The simplest case of one-variable and degree-one asks for the probability that a random integer is squarefree, which is well-known to be $6/\pi^2$. The one-variable degree-two case can also be solved by elementary methods. The one-variable degree-three case was solved by Hooley \cite{Hooley}. For homogeneous polynomials of two variables, the question is known up to degree 6 due to Greaves \cite{Greaves}. For non-homogeneous polynomials of two variables that factor completely into a product of linear factors over some extension of $\Q$, the question is known also up to degree 6 due to Hooley \cite{HooleyP}. Very recently, Kowalski \cite{Kow} proved the case where $F$ is a sum of at least $3$ cubic polynomials in different variables. The cases when $F$ is the discriminant of monic polynomials or when $F$ is the discriminant of general polynomials were solved by Bhargava-Shankar-Wang \cite{BSW_Sqfree,BSW_Sqfree2}. 

Conditionally, Granville \cite{Gra} showed that the one-variable case in general follows from the $abc$-conjecture and a bound on the error term was later obtained by Murty-Pasten \cite{MurPas}. Poonen \cite{Poo} generalized it to the multi-variable case where the variables are growing to infinity one by one.

Very little is known otherwise. In most cases, even the infinitude of squarefree values is open---the most famous example being $a^4+2.$ There is a handful of results on the number of prime values taken by polynomials: Friedlander-Iwaniec \cite{FI} for $a^4+b^2$; Heath-Brown \cite{HB} for $a^3+2b^3$ and its generalization by Heath-Brown-Moroz \cite{HBM} for irreducible binary cubic forms; and Health-Brown-Li \cite{HBL} for $a^2+p^4$ where $p$ is a prime.

In this paper, we consider for the first time the polynomial $a^4+b^3$. Our method in fact allows us to consider all polynomials of the form $\beta a^4+\alpha b^3$ for any fixed integers $\alpha$ and $\beta$. We prove:

\begin{theorem}\label{thm:main}
Let $\alpha$ and $\beta$ be fixed nonzero integers such that $\gcd(\alpha,\beta)$ is squarefree. Let $$N(X;\alpha,\beta) = \#\{(a,b)\in\Z^2\colon \max\{|a|^{1/3},|b|^{1/4}\}<X,\beta a^4+\alpha b^3 \mbox{ is squarefree}\}.$$ For any positive integer $m$, let $\rho_{\alpha,\beta}(m)=\#\{(a,b) \bmod{m} \colon m \mid \beta a^4+\alpha b^3\}$ and let $$C(\alpha,\beta) = \prod_p (1 - \rho_{\alpha,\beta}(p^2)p^{-4}).$$ Then
$$N(X;\alpha,\beta) = C(\alpha,\beta)X^7 + O_\epsilon(X^{6.992+\epsilon}).$$
The implied constant depends on $\alpha$ and $\beta$.
\end{theorem}

The case $\alpha=256$ and $\beta = -27$ is of special importance since $256b^3 - 27a^4$ is the discriminant of the quartic polynomial $x^4 + ax + b.$ An elementary calculation shows that $\rho_{256,-27}(p^2)$ equals $p^3$ for $p=2,3$; and equals $2p^2-p$ for $p\geq 5$. Therefore, we have:

\begin{theorem}\label{thm:main2}
When pairs $(a,b)$ of integers are ordered by $H(a,b) = \max\{|a|^{1/3},|b|^{1/4}\}$, the density of quartic polynomials of the form $x^4+ax+b$ having squarefree discriminant exists and is equal to $$\frac13\prod_{p\geq5}(1 - \frac{2}{p^2} + \frac{1}{p^3})$$ which is approximately $28.03\% $.
\end{theorem}

It is easy to see that the Euler product $C(\alpha,\beta)$ gives an upper bound for the desired density, if it exists, by applying the Chinese Remainder Theorem to more and more primes. As is standard in sieve theory, to demonstrate the lower bound, a ``tail estimate'' is required to show that there are not too many pairs $(a,b)$ of integers such that $\beta a^4+\alpha b^3$ is divisible by $m^2$ for some squarefree integer $m$. More precisely, we prove:

\begin{theorem}\label{thm:tailestimate}
Let $\alpha$ and $\beta$ be fixed nonzero integers such that $\gcd(\alpha,\beta)$ is squarefree. For any squarefree integer $m$, let 
$$
N_m(X;\alpha,\beta) = \#\{(a,b)\in\Z^2\colon |a|\leq X^3, |b|\leq X^4, m^2\mid \beta a^4 +\alpha b^3\}.
$$
Then for any positive real number $M$ and $\epsilon >0$,
\begin{equation}\label{eq:final}
\sum_{\substack{m>M\\ m\;\mathrm{ squarefree}
 }} N_m(X;\alpha,\beta) = O_\epsilon\left(\frac{X^{7+\epsilon}}{\sqrt{M}}\right) + O_\epsilon(X^{6.992+\epsilon})
\end{equation}
The implied constants depend on $\alpha$ and $\beta$.
\end{theorem}

We note that since the exponents $3$ and $4$ are coprime, it is enough to prove Theorem \ref{thm:tailestimate} for one choice of $\alpha$, $\beta$. Indeed, we have
$$-256\cdot 27\cdot \alpha^8\beta^3(\beta a^4+\alpha b^3) = 256(-3\alpha^3\beta b)^3 - 27(4\alpha^2\beta a)^4,$$
which implies that, $$N_m(X;\alpha,\beta) \leq N_m(c_{\alpha,\beta}X; \,256, -27)$$
for some constant $c_{\alpha,\beta}$ depending only on $\alpha,\beta.$ Hence the power saving bound \eqref{eq:final} for $\alpha=256$ and $\beta = -27$ implies it for all other $\alpha$ and $\beta.$ We simplify notation by writing $\Delta(a,b)$ for $256b^3 - 27a^4$. 

For any prime $p$ and pair $(a,b)$ of integers such that $p^2\mid\Delta(a,b)$, we say $p^2$ \emph{strongly divides} $\Delta(a,b)$ if $p^2\mid \Delta(a',b')$ for any integers $a'\equiv a\pmod{p}$ and $b'\equiv b\pmod{p}$; otherwise, we say $p^2$ \emph{weakly divides} $\Delta(a,b)$. Note in this case, for $p\geq5$, $p^2$ strongly divides $\Delta(a,b)$ if and only if $p\mid a$ and $p\mid b.$ For any squarefree integer $m$, let $\W_m^{(1)}$ (respectively $\W_m^{(2)}$) denote the set of pairs $(a,b)$ of integers such that $p^2$ strongly divides (respectively weakly divides) $\Delta(a,b)$ for every prime $p\mid m$. Then we prove:

\begin{theorem}\label{thm:weak}
For any positive real number $M$ and $\epsilon >0$, 
\begin{eqnarray}
\label{eq:W1}
{\rm (a)}\quad \#\bigcup_{\substack{m>M\\ m\;\mathrm{ squarefree}
 }}\{(a, b)\in\W_{m}^{(1)}\colon H(a, b)<X\} &=& O\Big(\frac{X^7}{M}\Big) + O\Big(X^4 \log X\Big);\\[.075in]
\label{eq:W22}
{\rm (b)}\quad
\#\bigcup_{\substack{m>M\\
m\;\mathrm{ squarefree}
}}\{(a, b)\in\W_{m}^{(2)}\colon H(a, b)<X\} &=& O_\epsilon\Big(X^{6.992+\epsilon}\Big) + O_\epsilon\left(\frac{X^{7+\epsilon}}{M}\right),
\end{eqnarray}
where the implied constants are independent of $M$ and $X$.
\end{theorem}

We now briefly describe our methods. Theorem \ref{thm:weak}(a) is immediate with the first term counting the contribution from $a\neq 0$ and the second term counting the contributiom from $a=0$. We devote the rest of the paper to proving Theorem \ref{thm:weak}(b). We follow the strategy of \cite{BSW_Sqfree} to embed $\W_{m}^{(2)}$ into the space $W$ of $4\times 4$ symmetric matrices. More precisely, let $A_0$ denote the $4\times 4$ matrix with $1$'s on the anti-diagonal and $0$'s elsewhere. The group
$G=\PSO(A_0)=\SO(A_0)/\langle\pm I\rangle$ acts on $W$ via the action $g\cdot B=gBg^t$ for $g\in G$
and $B\in W$.  Define the {\it invariant polynomial} of an element
$B\in W$ by $$f_B(x) = \det(A_0x - B).$$ Then $f_B$ is a monic quartic  polynomial. We extend the definition $H(a,b)$ to arbitrary monic quartic polynomials by $$H(x^4+c_1x^3 + c_2x^2 + c_3x + c_4) = \max\{|c_1|,|c_2|^{1/2},|c_3|^{1/3},|c_4|^{1/4}\}.$$ Define the discriminant and height of an element $B\in W$ by the discriminant and height of $f_B$, respectively. We then construct a map $$\sigma_m:\W_{m}^{(2)}\rightarrow \frac14 W(\Z)$$ with $f_{\sigma_m(a,b)} = x^4+ax+b$ as in \cite{BSW_Sqfree}, where $\frac14W(\Z)$ is the lattice of elements $B$ whose coefficients have denominators dividing $4$. 
It thus remains to count $G(\Z)$-orbits in $\frac14 W(\Z)$ that intersect the image of $\sigma_m$ for some squarefree $m>M$, and have height bounded by $X$. 


The space $W$ has several subspaces: $W_{00}$ consisting of $B\in W$ whose $(1,1)$- and $(1,2)$-entries are $0$; $W_{01}$ consisting of $B\in W$ whose $(1,1)$- and $(1,3)$-entries are $0$; and $W_0$ consisting of $B\in W$ whose $(1,1)$-entry is 0. The map $\sigma_p$ in fact lands in $W_{00}$, and so are guaranteed to be \emph{distinguished} in the sense of \cite{SW}. We obtain a bound of $O_\epsilon(X^{7+\epsilon}/M)$ for the distinguished cusps $W_{00}$ and $W_{01}$ and a bound of $O_\epsilon(X^{6+\epsilon})$ for the ``thick'' cusp $W_0$.

The main novelty of this paper is on counting the distinguished orbits in the main body. We use the circle method to handle the condition that the invariant polynomials have vanishing $x^2$-coefficients, combined with the Selberg sieve to impose the distinguished condition to obtain the desired power saving.

We remark that Heath-Brown's result \cite{HBd} on $k$-free values of the polynomial $n^d + c$ specializes to squarefree values of the cubic polynomial $n^3+c$ where $c$ is a constant. A major observation of \cite{HBd} is that counting triples $(n,s,t)$ with $n^3+c= s^2t$ when $n$ and $s$ are large and $c$ is fixed, is akin to counting points close to the projective curve $N^3 = S^2T$. The bigger $c$ is, which in our case can be as big as $n^3$, the worse the estimate gets. As such, we cannot patch the results of \cite{HBd} together to prove Theorem \ref{thm:main}. One does immediately obtain the infinitude of squarefree values.

This paper is organized as follows. In Section \ref{sec:embed}, we set up the embedding into $W$ and collect some results on the invariant theory for the action of $G$ on $W$, which allows us to reduce Theorem \ref{thm:weak}(b) to a result on counting $G(\Z)$-orbits in $\frac14W(\Z)$. In Section \ref{sec:prelim_count}, we apply Bhargava's averaging trick and count in the cusp. In Section \ref{sec:circle_count}, we use the circle method and the Selberg sieve to count in the main body. Finally, in Section \ref{sec:proof}, we show how Theorem \ref{thm:main} and Theorem \ref{thm:tailestimate} follow from Theorem \ref{thm:weak}.

%% file: embed2.tex
\section{Embedding into the space of $4\times 4$ symmetric matrices}\label{sec:embed}

Let $A_0$ be the $4\times 4$ matrix with $1$'s on the anti-diagonal and $0$'s elsewhere. The group $G=\PSO(A_0)=\SO(A_0)/\langle \pm I\rangle$ acts on the space $W$ of symmetric $4\times 4$ matrices via the action $g\cdot B = gBg^t$ for $g\in G$ and $B\in W$. The ring of polynomial invariants over $\C$ is freely generated by the coefficients of the invariant polynomial $f_B(x) = \det(A_0x - B)$, which is a monic quartic polynomial. Define $G$-invariant discriminant $\Delta(B)$ and height $H(B)$ of an element $B\in W$ by $\Delta(B) = \Delta(f_B)$ and $H(B) = H(f_B)$. We recall some of the arithmetic invariant theory for this representation. See \cite{SW,BSW_Sqfree} for more detail.

\subsection{Invariant theory for the representation $W$ of $G$}\label{sec:invar}

Let $k$ be a field of characteristic not $2$. For any monic quartic polynomial $f(x)\in k[x]$ such that $\Delta(f)\neq 0$, let $C_f$ denote the smooth hyperelliptic curve $y^2 = f(x)$ of genus $1$ and let $J_f$ denote its Jacobian (which is an elliptic curve). The stabilizer in $G(k)$ of an element $B\in W(k)$ with $f_B(x) = f(x)$ is naturally isomorphic to $J_f[2](k)$, which in turn is in bijection with the set of \emph{even factorization} of $f(x)$ over $k$. An even factorization of $f(x)$ over $k$ is an unordered pair $(g(x),h(x))$ of quadratic polynomials with $g(x)h(x)=f(x)$ such that either $g$ and $h$ are both defined over $k$, or they are (defined and) conjugate over a quadratic extension of $k$.

An element $B\in W(k)$ or its $G(k)$-orbit with $\Delta(B)\neq 0$ is said to be: $k$-\emph{soluble} if there exists a nonzero vector $v\in k^4$ such that 
\begin{equation}\label{eq:sol}
v^tA_0v = 0 = v^tBv;
\end{equation}
$k$-\emph{distinguished} if there exist linearly independent vectors $v,w\in k^4$ such that
\begin{equation}\label{eq:dis}
v^tA_0v = v^tBv = w^tA_0w = v^tA_0w = v^tBw = 0.
\end{equation}
Moreover, the set of $k$-lines $\Span(v)$ satisfying \eqref{eq:sol}, if nonempty, is in bijection with $J_{f_B}(k)$; and the set of $k$-flags $\Span(v)\subset\Span(v,w)$ satisfying \eqref{eq:dis}, if nonempty, is in bijection with $J_{f_B}[2](k).$
The set of $k$-soluble orbits with $f_B(x) = f(x)$ is in bijection with $J(k)/2J(k)$. The number of $k$-distinguished orbits with $f_B(x) = f(x)$ is $1$ if $f(x)$ has a linear factor over $k$ or if $f(x)$ admits a factorization of the form $g(x)h(x)$ where $g$ and $h$ are not rational over $k$ but are conjugate over a quadratic extension of $k$; and is $2$ otherwise.

Let $W_{00}$ denote the subspace of $W$ consisting of matrices $B$ whose $(1,1)$- and $(1,2)$-entries are $0$. Let $W_{01}$ denote the subspace of $W$ consisting of matrices $B$ whose $(1,1)$- and $(1,3)$-entries are $0$. Let $W_0$ denote the subspace of $W$ consisting of matrices $B$ whose $(1,1)$-entry is $0$. Then elements in $W_{00}(k)$ or $W_{01}(k)$ with nonzero discriminants are $k$-distinguished and elements in $W_0(k)$ with nonzero discriminants are $k$-soluble. A further polynomial invariant, called the $Q$-invariant, is defined on $W_{00}$ in \cite[\S3.1]{BSW_Sqfree}. For the case of $4\times 4$ matrices $B$, this is simply the $(1,3)$-entry $b_{13}$. The $Q$-invariant has the following important property:

\begin{proposition}\label{prop:Q-inv}
Let $B\in W_{00}(\Q)$ be an element whose invariant polynomial $f_B(x)$ is irreducible over $\Q$ and does not factor as $g(x)h(x)$ where $g$ and $h$ are conjugate over some quadratic extension of $\Q$. If $B'\in W_{00}(\Q)$ is any element that is $G(\Z)$-equivalent to $B$, then the $(1,3)$-entries of $B'$ and $B$ are equal up to sign. If $B'\in W_{01}(\Q)$ is any element that is $G(\Z)$-equivalent to $B$, then the $(1,2)$-entry of $B'$ equals the $(1,3)$-entry of $B$ up to sign.
\end{proposition}

\begin{proof} We prove the statement for $B'\in W_{01}(\Q)$. The statement for $W_{00}(\Q)$ follows by a similar argument (see also \cite[Proposition 3.1]{BSW_Sqfree}).

Let $\{e_1,e_2,e_3,e_4\}$ denote the standard basis for $\Q^4.$ Let $\gamma_0$ be the element of $\SO(A_0)(\Z[i])$ defined by $$\gamma_0(e_1) = ie_1,\quad \gamma_0(e_2)=ie_2,\quad \gamma_0(e_3) = -ie_3,\quad \gamma_0(e_4) = -ie_4,$$
where $i=\sqrt{-1}$ is a root to $x^2+1=0$. Then any $\gamma\in\PSO(A_0)(\Z)$ can either be lifted to some $\widetilde{\gamma}\in\SO(A_0)(\Z)$ or to $\gamma_0\widetilde{\gamma}\in\SO(A_0)(\Z[i])$ for some $\widetilde{\gamma}\in\SO(A_0)(\Z)$. 

Suppose $B'=\gamma B \gamma^t$ for some $\gamma\in \PSO(A_0)(\Z)$. Since $\gamma_0$ only scales $e_1$ and $e_3$, we see that there is some $\widetilde{\gamma}\in\SO(A_0)(\Z)$ such that $\Span(\widetilde{\gamma}^te_1) \subset \Span(\widetilde{\gamma}^te_1, \widetilde{\gamma}^te_3)$ is a flag satisfying \eqref{eq:dis} for $B$, and $\gamma$ lifts either to $\widetilde{\gamma}$ or to $\gamma_0\widetilde{\gamma}$. The assumption on $f_B$ implies that the flag $\Span(\widetilde{\gamma}^te_1) \subset \Span(\widetilde{\gamma}^te_1, \widetilde{\gamma}^te_3)$ coincides with the flag $\Span(e_1)\subset\Span(e_1,e_2).$ Hence, there are integers $\alpha_1,\alpha_2,\alpha_3$ such that
\begin{eqnarray*}
\widetilde{\gamma}^te_1 &=& \alpha_1 e_1,\\
\widetilde{\gamma}^te_3 &=& \alpha_2e_2 + \alpha_3e_1.
\end{eqnarray*}
Since $\widetilde{\gamma}\in\SO(A_0)(\Z)$, we must then have
\begin{eqnarray*}
\widetilde{\gamma}^te_4 &=& \alpha_1^{-1} e_4 - \alpha_3\alpha_1^{-1}\alpha_2^{-1}e_3,\\
\widetilde{\gamma}^te_2 &=& \alpha_2^{-1}e_3,
\end{eqnarray*}
with $\alpha_1=\pm1$ and $\alpha_2=\pm1.$ The $(1,2)$-entry $b'_{12}$ of $B'$ is then either $(\widetilde{\gamma}^te_1)^tB(\widetilde{\gamma}^te_2)$ or $(i\widetilde{\gamma}^te_1)^tB(i\widetilde{\gamma}^te_2)$. In both cases, we have $b'_{12}=\pm \alpha_1\alpha_2^{-1}e_1^tBe_3 = \pm b_{13}$.
\end{proof}

Let $U\simeq\A^2\backslash\{\Delta=0\}$ be the space of monic quartic polynomials of the form $x^4 + ax + b$ with nonzero discriminant. The next result shows that the number of elements of $U(\Z)$ failing the condition of Proposition \ref{prop:Q-inv} is negligible.

\begin{proposition}\label{prop:red}
The number of elements $f\in U(\Z)$ with $H(f)<X$ such that $f(x)$ is either reducible over $\Q$ or factors as $g(x)h(x)$ where $g$ and $h$ are conjugate over some quadratic extension of $\Q$ is $O(X^4\log X)$.
\end{proposition}

\begin{proof}
Throughout this proof, we use repeatedly the classical result that the sum $\sum_{|n|<X}d(n)$ of the divisor function is $O(X\log X)$ and that the sum $\sum_{|n|<X}\tau_3(n)$ of the triple-divisor function is $O(X\log^2X).$ See for example \cite[\S3.5]{Apo}.

Suppose first $f(x) = x^4 + ax + b$ has a linear factor $x-r$ over $\Q$. Then there are $O(X^3)$ of them with $b=0$. Suppose now $b\neq 0$. Then since $r\mid b$, we get $O(X^4\log X)$ choices for the pair $(r,b)$, which uniquely determines $a$. Hence, there are $O(X^4\log X)$ such $f(x)$ with a linear factor.

Next we consider the case where $f(x) = x^4 + ax + b$ does not have a linear factor but factors as $(x^2 + cx + d)(x^2 - cx + e)$ over $\Q$. Then $de = b\neq 0$, which gives us $O(X^4\log X)$ choices for the triple $(d,e,b)$. Comparing the $x^2$-coefficients gives $c^2 = d + e$, and so $c$ is determined given $d$ and $e$. Comparing the $x$-coefficients then uniquely determines $a$. Hence, there are $O(X^4\log X)$ such $f(x)$ that factors as a product of two irreducible quadratic polynomials.

Finally, we consider the case where $f(x) = x^4 + ax + b$ is irreducible over $\Q$ but factors as $$(x^2 + e_1\sqrt{d}x + \frac{c_2 + e_2\sqrt{d}}{2})(x^2 - e_1\sqrt{d}x + \frac{c_2 - e_2\sqrt{d}}{2})$$ over the ring of integers in $\Q(\sqrt{d})$ for some $d$. If $a=0$, then we have $O(X^4)$ choices for $b$. Suppose now $a\neq 0$. Comparing the $x$-coefficients gives $e_1e_2d = a$. Hence there are $O(X^3\log^2 X)$ choices for the tuple $(e_1,e_2,d,a)$. Comparing the $x^2$-coefficients gives $c_2 -e_1^2d=0$, and so $c_2$ is determined given $e_1$ and $d$. Comparing the constant terms then uniquely determines $b$. Hence, there are $O(X^3\log^2 X)$ such $f(x)$ that factors into conjugate quadratic polynomials over some quadratic extension of~$\Q$. 
\end{proof}

We end this section with a bound on distinguished elements over finite fields, which will be used in the Selberg sieve in Section \ref{sec:circle_count}.

\begin{proposition}\label{prop:F_p-dist}
Let $p\geq 7$ be a prime. Then the number $d_p$ of elements $B\in W(\F_p)$ with $f_B\in U(\F_p)$ and are not $\F_p$-distinguished satisfies $$\frac{1}{16}p^8+O(p^7) \leq d_p \leq \frac{3}{4}p^8 + O(p^7).$$
\end{proposition}

\begin{proof}
Over the finite field $\F_p$, every orbit is $\F_p$-soluble. Moreover, for any monic quartic polynomial $f(x)\in\F_p[x]$ with nonzero discriminant, the number $\#J_f(\F_p)/2J_f(\F_p)$ of $\F_p$-orbits with invariant polynomial $f$ equals the size $\#J_f[2](\F_p)$ of any stabilizer with invariant polynomial $f$. Hence, the number of $B\in W(\F_p)$ with $f_B = f$ equals $\#G(\F_p)=p^2(p^2-1)^2.$ There are $p^2 + O(p)$ polynomials $f\in U(\F_p)$ and so a total of $p^8 + O(p^7)$ elements $B\in W(\F_p)$ with $f_B\in U(\F_p)$. Moreover, for any $f\in U(\F_p)$, there is at least one $\F_p$-distinguished orbit with stabilizer having size at most $4$. Hence, we have the upper bound $d_p \leq \frac34p^8+O(p^7)$.

Consider next quartic polynomials of the form $$g_{a,b}(x):=(x-a)(x-b)(x^2 + (a+b)x + (a^2 + ab + b^2))\in U(\F_p).$$
Since $g_{a,b}(x)$ has a linear factor, there is only one distinguished orbit with invariant $g_{a,b}$. Moreover, we have $2\leq \#J_{g_{a,b}}[2](\F_p)\leq 4$. Hence, there is at least one non-distinguished orbit of size at least $|G(\F_p)|/4.$ It remains to count the number of such $g_{a,b}(x)$ with nonzero discriminant, which is equivalent to requiring that $a\neq b$, that $a$ is not a root of the quadratic factor, and that the quadratic factor has nonzero discriminant. In other words, we have $a\neq b$, $3(a+b/3)^2 + (2/3)b^2\neq0$ and $3(a+b/3)^2 + (8/3)b^2\neq0$. Given any $b$, there are at least $p-5$ choices for $a$. Finally, given any $g_{a,b}$ with nonzero discriminant, we see that $g_{a,b} = g_{a',b'}$ if and only if $x^2 + (a+b)x + (a^2+ab+b^2)=(x-a')(x-b')$ or $(x-a)(x-b)=(x-a')(x-b')$, as any other possibility contradicts $\Delta(g_{a,b})\neq0$. Hence, there are at least $p(p-5)/4$ quartic polynomials with nonzero discriminant of the form $g_{a,b}$ for some $a,b\in\F_p.$ Therefore, we have at least $\frac{1}{16}p^8+O(p^7)$ non-distinguished elements $B$ in $W(\F_p)$ with $f_B\in U(\F_p).$ 
\end{proof}

\subsection{Embedding $\W_m^{(2)}$ into $\frac14W(\Z)$}

In light of Proposition \ref{prop:red}, it is sufficient to prove Theorem \ref{thm:weak} with $\W_m^{(2)}$ replaced by the set $\W_m^{(2),\irr}$ of pairs $(a,b)\in\W_m^{(2)}$ such that $f_{a,b}(x):=x^4+ax+b$ is irreducible and does not factor into a product of quadratic polynomials conjugate over some quadratic extension of $\Q$. We prove some preliminary results in order to use the map $\sigma_m$ defined in \cite[\S3.2]{BSW_Sqfree}.

Fix $(a,b)\in\W_m^{(2)}$ and fix any prime $p\mid m$. For any $(a',b')\in\Z^2$ with $a'\equiv a\pmod{p}$ and $b'\equiv b\pmod{p}$, we have $f_{a',b'}(x)\equiv f_{a,b}(x)\pmod{p}$. Since $p^2$ weakly divides $\Delta(a,b)$, we see that $f_{a,b}(x)$ has a unique double root mod $p$. Let $r\in \Z$ be an integer such that $f_{a,b}(x+r)=x^4 + b_1x^3 + b_2x^2 + b_3x + b_4$ with $p\mid b_3$ and $p\mid b_4.$ We claim that $p^2\mid b_4$. Note the discriminant of a quartic polynomial is of the form
$$\Delta(x^4 + b_1x^3 + b_2x^2 + b_3x + b_4) = b_4\Delta'(b_1,b_2,b_3,b_4) + b_3^2\Delta(x^3 + b_1x^2 + b_2x + b_3)$$
where $\Delta'$ is some polynomial with integer coefficients. Suppose for a contradiction that $p^2\nmid b_4$. Then, since $p^2\mid\Delta(f_{a,b}(x+r))$, we have $p\mid \Delta'(b_1,b_2,b_3,b_4).$ Hence, $p^2\mid \Delta(g(x))$ for any monic quartic polynomial $g(x)$ congruent to $f_{a,b}(x+r)$. Now for any $(a',b')\in\Z^2$ with $a'\equiv a\pmod{p}$ and $b'\equiv b\pmod{p}$, we have $f_{a',b'}(x+r)\equiv f_{a,b}(x+r)\pmod{p}$ and so $p^2\mid \Delta(f_{a',b'}(x+r))$. Since $\Delta(f_{a',b'}(x+r)) = \Delta(a',b')$, this contradicts the assumption that $p^2$ weakly divides $\Delta(a,b).$

By the Chinese Remainder Theorem, there exists an integer $r$ such that $f_{a,b}(x+r) = x^4 + c_1x^3 + c_2x^2 + mc_3x + m^2c_4$ for some integers $c_1,c_2,c_3,c_4$. Consider the following matrix:
$$B(c_1,c_2,c_3,c_4) = 
\begin{pmatrix}
0&0&m&0\\
0&1&-c_1/2&0\\
m&-c_1/2&c_1^2/4-c_2&-c_3/2\\
0&0&-c_3/2&-c_4
\end{pmatrix}.$$
A direct computation shows that $f_{B(c_1,c_2,c_3,c_4)} = x^4 + c_1x^3 + c_2x^2 + mc_3x + m^2c_4.$ We now set $\sigma_m(a,b) = B(c_1,c_2,c_3,c_4)+rA_0\in\frac14W(\Z).$ Then $f_{\sigma_m(a,b)} = f_{a,b}(x)$. Note in fact that the image of $\sigma_m$ lies inside $W_{00}(\Q)$ and the $(1,3)$-entry of any element in the image of $\sigma_m$ is $m$. We summarize the above in the following theorem.

\begin{theorem}\label{thm:lift}
Let $m$ be any squarefree integer. There is a map $\sigma_m:\W_m^{(2)}\rightarrow \frac14W(\Z)$ such that $f_{\sigma_m(a,b)} = f_{a,b}(x)$ for any $(a,b)\in\W_m^{(2)}$. Moreover, the $(1,3)$-entry (respectively the $(1,2)$-entry) of any element in $W_{00}(\Q)$ (respectively $W_{01}(\Q)$) that is $G(\Z)$-equivalent to some element in $\sigma_m(\W_m^{(2),\irr})$ equals $p$ in absolute value. 
\end{theorem}

%% file: prelim_count2.tex
\section{Averaging and counting in the cusp}\label{sec:prelim_count}

Let $\L_M$ denote the set of elements in $\frac14W(\Z)$ that are $G(\Z)$-equivalent to some elements in $\sigma_m(\W_{m,M}^{(2),\irr})$ for some prime $p>M$. Write $N(\L_M, X)$ for the number of $G(\Z)$-orbits in $\L_M$ having height at most $X$. Theorem \ref{thm:lift} implies that Theorem \ref{thm:weak} follows from the following bound on $N(\L_M, X)$.

\begin{theorem}\label{thm:count}
For any positive real number $M$ and any $\epsilon >0$, we have
\begin{equation}\label{eq:countN}
N(\L_M, X) = O_\epsilon\Big(X^{6.992+\epsilon}\Big) + O_\epsilon\left(\frac{X^{7+\epsilon}}{M}\right).
\end{equation}
\end{theorem}

In Section \ref{sec:prelimsetup}, we recall the set up in \cite{SW} for counting $G(\Z)$-orbits in $\frac14W(\Z)$ and divide up a fundamental domain $\FF$ for the left-multiplication action of $G(\Z)$ on $G(\R)$ into the main body, the thick cusp, and the distinguished cusps. In Section \ref{sec:prelimcusp}, we obtain bounds for the contribution from the thick cusp and the distinguished cusps. Finally in Section \ref{sec:circle_count}, we obtain bounds for the contribution from the main body and complete the proof of Theorem \ref{thm:count}.

\subsection{Counting $G(\Z)$-orbits in $\frac14W(\Z)$}\label{sec:prelimsetup}

The counting problem for the representation $W$ of $G$ is studied in \cite{SW}. In this section, we recall some of the set up and results of \cite{SW}.

Let $R$ be a fundamental domain for the action of $G(\R)$ on the elements of $W(\R)$ having nonzero discriminant and height bounded by $1$ as constructed in \cite[\S4.1]{SW}. Let $\FF$ be a fundamental set for the left-multiplication action of $G(\Z)$ on $G(\R)$ obtained using the Iwasawa decomposition of $G(\R)$. More explicitly, we have
$$
G(\R)=N(\R)TK,
$$ where $N$ is a unipotent group consisting of lower triangular matrices, $K$ is compact, and $T$ is the split torus of $G$ given by
\begin{equation*}
T=
\left\{\left(\begin{array}{cccc}
 t_1^{-1}&&&\\
 & t_2^{-1} &&\\
 &&t_2 &\\
 &&&t_1  
\end{array}\right)\right\}.
\end{equation*}
We also make the following change of variables: set $$s_1 = t_1/t_2,\quad s_2 = t_1t_2.$$
We denote an element of $T$ with coordinates $t_i$
(resp.\ $s_i$) by $(t)$ (resp.\ $(s)$). We may take $\FF$ to be contained in a {\it Siegel set}, i.e., contained in $N'T'K$,
where $N'$ consists of elements in $N(\R)$ whose entries are
absolutely bounded and $T'\subset T$ consists of elements in $(s)\in
T$ with $s_1\geq c$ and $s_2\geq c$ for some positive constant $c$.

For any $h\in G(\R)$, since $\FF h$ remains a fundamental domain for
the action of $G(\Z)$ on $G(\R)$, the set $(\FF h)\cdot (XR)$ (when
viewed as a multiset) is a finite cover of a fundamental domain for
the action of $G(\Z)$ on the elements in $W(\R)$ with nonzero
discriminant and height bounded by $X$. The degree of the cover
depends only on the size of stabilizer in $G(\R)$ and is thus
absolutely bounded by $4$. The presence of these stabilizers is
in fact the reason we consider $(\FF h) \cdot (XR)$ as a multiset. Hence,
we have
\begin{equation}\label{eqoddfundv}
  N(\L_M,X)\ll \#\big\{\big((\FF h)\cdot (XR)\big)\cap\L_M\big\}.
\end{equation}
Let $\GG_1$ be a compact left $K$-invariant set in $G(\R)$ which is the
closure of a nonempty open set.
Averaging \eqref{eqoddfundv} over $h\in \GG_1$ and exchanging the
order of integration as in \cite[Theorem 2.5]{BS2}, we obtain
\begin{equation}
  N(\L_M,X)\ll \int_{\gamma\in\FF} \#\big\{\big((\gamma \GG_1)\cdot (XR)\big)\cap\L_M\big\}
  d\gamma,
\end{equation}
where the implied constant depends only on $\GG_1$ and $R$, and where $d\gamma$ is a Haar measure on $G(\R)$ given by
$$
d\gamma=dn\,s_1^{-1}s_2^{-1}d^\times s\,dk,
$$ where $dn$ is a Haar measure on the unipotent group $N(\R)$,
$dk$ is a Haar measure on the compact group~$K$, and $d^\times s=s_1^{-1}ds_1\,s_2^{-1}ds_2$ is the standard Haar measure on $\mathbb{G}_m^2$ (see \cite[(20)]{SW}).

Since $s_i\geq c$ for every $i$, there exists a compact subset $N''$ of
$N(\R)$ containing $(t)^{-1}N'\,(t)$ for all $t\in T'$. Since $N''$, $K$, $\GG_1$ are compact and $R$ is bounded, the set $E=N''K\GG_1R$ is bounded. Then we have
\begin{equation}\label{eq:REe}
N(\L_M,X)\ll \int_{s_i\gg 1} \#\big\{\big((s)\cdot XE\big)\cap\L_M\big\}s_1^{-1}s_2^{-1}d^\times s.
\end{equation}

We denote the coordinates of $W$ by $b_{ij}$, for $1\leq
i\leq j\leq 4$, and we denote the $T$-weight of a coordinate $\alpha$ on $W$ by $w(\alpha)$.
We compute the weights of the coordinates $b_{ij}$ to be
$$
\begin{array}{rclrclrclrcl}
w(b_{11}) &=& s_1^{-1}s_2^{-1},& w(b_{12}) &=& s_2^{-1},& w(b_{13}) &=& s_1^{-1},& w(b_{14}) &=& 1,\\
&&&w(b_{22}) &=& s_1s_2^{-1},& w(b_{23}) &=& 1,& w(b_{24}) &=& s_1, \\
&&&&&&w(b_{33}) &=& s_1^{-1}s_2,& w(b_{34}) &=& s_2,\\
&&&&&&&&&w(b_{44}) &=& s_1s_2.\\
\end{array}
$$
Then the $(i,j)$-entry of any $B\in (s)\cdot XE$ is bounded by $c_0Xw(b_{ij})$, where $c_0>0$ is a constant depending only on $\GG_1$ and $R$.

We define two distinguished cusps: $T_{00}\subset T'$ consisting of elements $(s)$ such that $c_0Xw(b_{11}) < 1/4$ and $c_0Xw(b_{12})<1/4$; and $T_{01}\subset T'$ consisting of elements $(s)$ such that $c_0Xw(b_{11}) < 1/4$ and $c_0Xw(b_{13})<1/4$. We define the thick cusp $T_0$ to be the subset of $T'$ consisting of elements $(s)$ such that $c_0Xw(b_{11}) < 1/4$,  $c_0Xw(b_{12})\geq1/4$, and $c_0Xw(b_{13})\geq1/4$. We define the main body $T''$ to be the complement $T'\backslash(T_{00}\cup T_{01}\cup T_0).$ Then for any $(s)\in T_{00}$, we have $\big((s)\cdot XE\big) \cap \frac14W(\Z)\subset W_{00}(\Q)$; for any $(s)\in T_{01}$, we have $\big((s)\cdot XE\big) \cap \frac14W(\Z)\subset W_{01}(\Q)$; and for any $(s)\in T_0$. we have $\big((s)\cdot XE\big) \cap \frac14W(\Z)\subset W_{0}(\Q)$.

Since elements in $\L_M$ have invariant polynomials in $U$, we express the conditions of the invariant polynomial having vanishing $x^3$- and $x^2$-coefficients in terms of the coordinates $b_{ij}$. The $x^3$-coefficient is the anti-trace, and so we have
$$b_{23} = -b_{14}.$$
After replacing $b_{23}$ by $-b_{14}$, we see that the $x^2$-coefficient is the following quadratic form:
$$q(b_{ij}) := -b_{11}b_{44} - b_{22}b_{33} - 2b_{12}b_{34} - 2b_{13}b_{24} - 2b_{14}^2.$$

\subsection{Counting in the cusps}\label{sec:prelimcusp}

In this section, we compute the contribution to \eqref{eq:REe} for $(s)\in T_{00}$, $(s)\in T_{01}$ and for $(s)\in T_0$.

\begin{proposition}\label{prop:cusp}
For any positive real number $M$ and $\epsilon >0$, we have
\begin{eqnarray}
\label{eq:distc0}\int_{(s)\in T_{00}} \#\big\{\big((s)\cdot XE\big)\cap\L_M\big\}\,s_1^{-1}s_2^{-1}d^\times s &=& O_\epsilon\Big(\frac{X^{7+\epsilon}}{M}\Big),\\
\label{eq:distc1}\int_{(s)\in T_{01}} \#\big\{\big((s)\cdot XE\big)\cap\L_M\big\}\,s_1^{-1}s_2^{-1}d^\times s &=& O_\epsilon\Big(\frac{X^{7+\epsilon}}{M}\Big),\\
\label{eq:thickc}\int_{(s)\in T_{0}} \#\big\{\big((s)\cdot XE\big)\cap\L_M\big\}\,s_1^{-1}s_2^{-1}d^\times s &=& O_\epsilon\Big(X^{6+\epsilon}\Big).
\end{eqnarray}
\end{proposition}

\begin{proof}
Consider first the distinguished cusp $T_{00}$. In this case, any element in $\big((s)\cdot XE\big)\cap \L_M$ is an element in $W_{00}(\Q)$ that is $G(\Z)$-equivalent to some element in $\sigma_p(\W_{p,M}^{(2),\irr})$ for some $p>M$. Hence, by Theorem \ref{thm:lift}, we have $|b_{13}|>M$ for any element $B \in \big((s)\cdot XE\big)\cap \L_M$. In other words, we have $Xs_1^{-1}\gg M.$ Moreover, note that if $B\in W$ with $b_{11}=b_{12}=b_{13}=0$ or $b_{11}=b_{12}=b_{22}=0$, then $\Delta(B) = 0$. Hence we may assume that $Xs_1s_2^{-1}\gg 1$. Let $T'_{00}$ denote the subset of $T_{00}$ consisting of elements $(s)$ with $Xs_1^{-1}\gg M$ and $Xs_1s_2^{-1}\gg 1$. Note we also have $s_1\ll X$ and $s_2\ll X^2$ for $(s)\in T'_{00}$.

The quadratic form $q(b_{ij})$ when restricted to $W_{00}$ simplifies to $q_1(b_{ij}) = -b_{22}b_{33}-2b_{13}b_{24}-2b_{14}^2.$ Hence, we have
\begin{eqnarray*}
\#\big((s)\cdot XE \cap \L_M\big) &\ll_\epsilon& \big((Xw(b_{14}))^{1+\epsilon}(Xw(b_{22}) + Xw(b_{33})) + (Xw(b_{14})Xw(b_{13})Xw(b_{24}))^{1+\epsilon}\big)Xw(b_{34})Xw(b_{44})\\
&\ll& X^{4+\epsilon}s_1^2s_2 + X^{4+\epsilon}s_2^3 + X^{5+\epsilon}s_1s_2^2\\
&\ll&X^{4+\epsilon}s_1^2s_2+X^{5+\epsilon}s_1s_2^2.
\end{eqnarray*}
Integrating these two terms separately gives
\begin{eqnarray*}
\int_{(s)\in T'_{00}} X^{4+\epsilon}s_1^2s_2s_1^{-1}s_2^{-1}d^\times s &=& \int_{(s)\in T'_{00}} X^{4+\epsilon}s_1d^\times s \quad\!\ll\quad\! \frac{X^{5+\epsilon}\log X}{M},\\
\int_{(s)\in T'_{00}} X^{5+\epsilon}s_1s_2^2s_1^{-1}s_2^{-1}d^\times s &=& \int_{(s)\in T'_{00}} X^{5+\epsilon} s_2d^\times s \quad\!\ll\quad\!\int_{(s)\in T'_{00}} X^{6+\epsilon} s_1d^\times s \quad\!\ll\quad\! \frac{X^{7+\epsilon}\log X}{M}.
\end{eqnarray*}
The integral over the other distinguished cusp $T_{01}$ has the same bound via the same analysis with $s_1$ and $s_2$ switched.

Finally, we consider the thick cusp $T_0$. In this case, we have $Xs_1^{-1}\gg 1$ and $Xs_2^{-1}\gg 1$. The quadratic form $q(b_{ij})$ when restricted to $W_0$ simplifies to $q_2(b_{ij})=- b_{22}b_{33} - 2b_{12}b_{34} - 2b_{13}b_{24} - 2b_{14}^2.$ The above analysis shows that the number of choices for $(b_{22},b_{33},b_{13},b_{24},b_{14})$ such that $q_1(b_{ij})=0$ is $O_\epsilon(X^{2+\epsilon}s_1s_2^{-1} + X^{2+\epsilon}s_1^{-1}s_2 + X^{3+\epsilon})$. Multiplying it by $(Xw(b_{12})+Xw(b_{34}))Xw(b_{44})$ gives a bound of $O_\epsilon(X^{4+\epsilon}s_1^2s_2 + X^{4+\epsilon}s_2^3 + X^{5+\epsilon}s_1s_2^2)$ for the number of $B\in \big((s)\cdot XE\big)\cap\L_M$ with $q_1(b_{ij})=0$. The contribution from $q_1(b_{ij})\neq 0$ is $$O_\epsilon\big((Xw(b_{22})Xw(b_{33})Xw(b_{13})Xw(b_{24})Xw(b_{14}))^{1+\epsilon}Xw(b_{44})\big)=O_\epsilon(X^{6+\epsilon}s_1s_2).$$ Using the bound $s_1\ll X$ and $s_2\ll X$, we have
$$\#\big\{\big((s)\cdot XE\big) \cap \L_M\big\} \ll_\epsilon X^{6+\epsilon}s_1s_2.$$
Multiplying by $s_1^{-1}s_2^{-1}$ and integrating then give the desired bound \eqref{eq:thickc}.
\end{proof}

For the main body $T''$, we have $Xs_1^{-1}s_2^{-1}\gg 1.$ Since both $s_1$ and $s_2$ are bounded below by some absolute constant, we still have the bound $s_1\ll X$ and $s_2\ll X$. The above analysis gives a bound of 
\begin{eqnarray*}
&&O_\epsilon\Big(\big((X^{2+\epsilon}s_1s_2^{-1} + X^{2+\epsilon}s_1^{-1}s_2 + X^{3+\epsilon})(Xw(b_{12})+Xw(b_{34}))+X^{5+\epsilon}\big)(Xw(b_{11})+Xw(b_{44}))\Big)\\
&=& O_\epsilon\Big((X^{3+\epsilon}s_1^{-1}s_2^2 + X^{4+\epsilon}s_2+X^{5+\epsilon})Xs_1s_2\Big)\\
&=& O_\epsilon\Big(X^{6+\epsilon}s_1s_2\Big)
\end{eqnarray*}
for the number of $B\in \big((s)\cdot XE\big)\cap \L_M$ with $q_2(b_{ij})=0$. Multiplying by $s_1^{-1}s_2^{-1}$ and integrating give a bound of $O_\epsilon(X^{6+\epsilon})$. It remains to consider the contribution to the main body integral from the number of $B\in \big((s)\cdot XE\big)\cap \L_M$ with $q_2(b_{ij})\neq 0$. We have a trivial bound of $O_\epsilon(X^{7+\epsilon})$ for the number of such $B$. 

For any positive real number $\delta$, let $T''_\delta$ denote the subset of $T''$ where $s_1\gg X^{\delta}$ or $s_2\gg X^{\delta}$. Then we have
\begin{equation}\label{eq:triv}
\int_{(s)\in T''_\delta} \#\big\{\big((s)\cdot XE\big)\cap \L_M\big\}\, s_1^{-1}s_2^{-1}d^\times s = O_\epsilon(X^{7-\delta+\epsilon}).
\end{equation}

Therefore, it remains to consider the main body integral under the additional assumption that $s_1\ll X^{\delta}$ and $s_2\ll X^{\delta}$ where $\delta$ is some small enough positive real number.

%% file: circle_count.tex
\section{Counting in the main body using the circle method}\label{sec:circle_count}

In this section, we consider the contribution to \eqref{eq:REe} from the main body under the additional assumption that $s_1\ll X^{\delta}$ and $s_2\ll X^{\delta}$. Let $V\simeq\A^9$ denote the subspace of $W$ cut out by $b_{14}=-b_{23}$. Since scaling an element $B\in V(\Q)$ by $4$ does not affect the vanishing of $q(B)$ or whether it is $\Q$-distinguished, it is enough to count points in a box in $V(\Z)$ defined by $|b_{ij}|\leq 4c_0Xw(b_{ij})$. The assumption on $s_1$ and $s_2$ implies that $4c_0Xw(b_{ij})=O(X^{1+2\delta})$ for all $i,j$. The goal of this section is to prove the following theorem:

\begin{theorem}\label{thm:circlemain}
Let $\delta < 0.01$ be a positive real number. Let $\B$ be a box in $V(\R)$ defined by $|b_{ij}|\leq X_{ij}$ for $(i,j)=(1,1),$ $(1,2),$ $(1,3),$ $(1,4),$ $(2,2),$ $(2,4),$ $(3,3),$ $(3,4),$ $(4,4)$ where $X_{ij}$ are real numbers satisfying $c_1^{-1}X^{1-2\delta}\leq X_{ij}\leq c_1X^{1+2\delta}$, $X_{14}=c_2 X$ and
$$X_{11} X_{44} = X_{22} X_{33} = X_{12} X_{34} = X_{13} X_{24} = c_2^2 X^2,$$ for some positive constants $c_1,c_2.$
Let $N^\dist_q(\B)$ denote the number of $\Q$-distinguished elements $B\in \B\cap V(\Z)$ with $q(B)=0$
. Then
\begin{equation}\label{eq:circlemain}
N_q^\dist(\B) = O_\epsilon\left(X^{\frac{209}{30} + \frac{137}{45}\delta + \epsilon}\right).
\end{equation}
\end{theorem}

Multiplying the bound \eqref{eq:circlemain} by $s_1^{-1}s_2^{-1}$ and integrating over $1\ll s_1,s_2\ll X^\delta$, combining with \eqref{eq:triv}, \eqref{eq:REe} and Proposition \ref{prop:cusp}, and setting $\delta = 3/364$ then complete the proof of Theorem \ref{thm:count}.

\vspace{10pt}
We will prove Theorem \ref{thm:circlemain} from Theorem \ref{thm:circlep} below by applying a Selberg sieve.

\begin{theorem}\label{thm:circlep}
Let $m$ be an odd squarefree positive integer with $m \ll X^{1/3}$.
Let $B_0 \in V(\Z)$ be an element such that $m \mid q(B_0)$ and $B_0$ is nonzero modulo $p$ for each prime factor $p$ of $m$. Let $N_q(\B; m, B_0)$ denote the number of $B\in \B$ such that $B\equiv B_0\pmod{m}$ and $q(B)=0$.

For each $r \geq 1$, set
\begin{equation}\label{eq:Cq}
    C_q(r) = \frac{1}{r^9} \sum_{\substack{0 \leq a < r \\ \gcd(a, r) = 1}} \sum_{B \bmod{r}} e\left(\frac{a}{r} q(B)\right)
\end{equation}
Define the singular series
\begin{equation}\label{eq:Sq}
    \kS(q) = \sum_{r \geq 1} C_q(r), 
\end{equation}
    and for each prime $p$, the series
\begin{equation}\label{eq:Sqp}
    \kS(q; p) = \sum_{\ell \geq 0} C_q(p^\ell).
\end{equation}
Define the singular integral
\begin{equation}\label{eq:Vq}
    \kS_\infty(\B;q) = \int_{\R} \int_{\B} e(\theta q(B))\, dB \; d\theta,
\end{equation}
    where $e(x) = e^{2\pi ix}$.
Then,
\begin{equation}\label{eq:circlep}
N_q(\B; m, B_0) = \frac{1}{m^8} \left(\prod_{p \mid m} \kS(q; p)^{-1}\right) \kS(q) \kS_\infty(\B; q)  + O\left(\frac{X^{6.85(1 + 2\delta)}}{m^{5.5}} \log X\right),
\end{equation}
    with the implied constant being absolute.
All the series defined above converge absolutely and they have positive value.
\end{theorem}

\subsection{Proof of Theorem \ref{thm:circlep} using the circle method}

Fix an odd squarefree $m$.
For any $\alpha \in [0, 1]$, let \[ S_\B(\alpha; m, B_0) = \sum_{\substack{B \in \B \cap V(\Z) \\ B \equiv B_0 \bmod{m}}} e\left(\frac{\alpha}{m} q(B)\right). \]
Then,  \[ N_q(\B; m, B_0) = \int_0^1 S_\B(\alpha; m, B_0) d\alpha. \]

Let $r_1$ and $r_2$ be positive real numbers, to be picked later, with $r_1 \ll \frac{X^{1 - 2 \delta}}{m}$ and $r_2 \gg X^{1 + 2 \delta}$.
Split the interval $[0, 1]$ into the major arc $\kM$ and the minor arc $\km = [0, 1] \setminus \kM$, where \[ \kM = \left\{\alpha : \left|\alpha - \frac a r\right| \leq \frac{1}{rr_2}, \gcd(a, r) = 1, 0 \leq a < r \leq r_1\right\}. \]

\subsubsection{Major arc estimate}

We estimate first the major arc integral \[ \int_\kM S_\B(\alpha; m, B_0) d\alpha = \sum_{r \leq r_1} \sum_{\substack{0 \leq a < r \\ \gcd(a, r) = 1}} \int_{|\theta| \leq \frac{1}{rr_2}} S_\B\left(\frac{a}{r} + \theta; m, B_0\right) d\theta. \]
Fix some $\alpha = \frac{a}{r} + \theta \in \kM$, where $|\theta| \leq \frac{1}{rr_2}$.
We have 
\begin{eqnarray*}
S_\B(\alpha; m, B_0) &=& \sum_{\substack{B_1 \bmod{rm} \\ B_1 \equiv B_0 \bmod{m}}} e\left(\frac{a}{rm} q(B_1)\right)\sum_{\substack{B \in \B \cap V(\Z) \\ B \equiv B_1 \bmod{rm}}} e\left(\frac{\theta}{m} q(B)\right)\\
&=&\sum_{\substack{B_1 \bmod{rm} \\ B_1 \equiv B_0 \bmod{m}}} e\left(\frac{a}{rm} q(B_1)\right)\sum_{B' \in \B' \cap V(\Z)} e\left(\frac{\theta}{m} q(rmB'+B_1)\right),
\end{eqnarray*}
where $\B'=\{B'\in V(\R)\colon rmB' + B_1\in \B\}$ is another box.
To compute the exponential sum over a box, we use the following result from \cite[Proposition 8.7]{IK}.

\begin{lemma}\label{lem:IK-prop8.7}
Let $f(x)$ be a real function on an interval $[a,b]$ such that $|f'(x)| \leq \frac{1}{2}$ for all $x\in (a,b)$. Suppose further that $f''(x) \geq 0$ on $(a,b)$ or that $f''(x)\leq 0$ on $(a,b)$.
Then, \[ \sum_{a < n < b} e(f(n)) = \int_a^b e(f(x))\, dx + O(1), \] with the implied constant being absolute.
\end{lemma}
We note that \cite[Proposition 8.7]{IK} requires that $f''(x)>0$, but the same proof applies when $f''(x)\geq0$ or when $f''(x)\leq 0.$ The following multivariable version also follows immediately.
\begin{lemma}\label{lem:IK8.7multi}
Let $f(x_1,\ldots,x_\ell)$ be a real function on a box $\RR=\prod_i [a_i,b_i]$ such that $|\frac{\partial f}{\partial x_i}(x)|\leq\frac12$ on $\RR$ for all $i=1,\ldots,\ell$. Suppose for any $i=1,\ldots,\ell$ and for any fixed $x_j\in(a_j,b_j)$ for all $j\neq i$, the second partial derivative $\frac{\partial^2 f}{\partial x_i^2}(x)$ as a function of $x_i$ is either non-negative on $(a_i,b_i)$ or non-positive on $(a_i,b_i)$. Then
$$\sum_{n\in\RR\cap \Z^\ell} e(f(n)) = \int_\RR e(f(x))\,dx + O(\max\{\Vol(\bar{\RR}),1\}),$$
where $\Vol(\bar{\RR})$ denotes the greatest $d$-dimensional volume of any projection of $\RR$ onto a coordinate subspace obtained by equating $\ell-d$ coordinates to zero, where $d$ takes all values from $1$ to $\ell-1$.
The implied constant depends only on $\ell$.
\end{lemma}

We apply Lemma \ref{lem:IK8.7multi} to the box $\B'$ and the quadratic polynomial $f(b_{ij})=\frac{\theta}{m}q(rmB'+B_1)$ viewed as a function in the coordinates of $B'$. The partial derivative of $f$ with respect to $b_{ij}$ equals $\theta r \frac{\partial q}{\partial b_{ij}}(rmB'+B_1)$ which is bounded by $c_3c_1\theta r X^{1+2\delta}$ where $c_3$ is a constant depending only on $q$ (and equals $2$ in this case). Hence, we can bound the first order partial derivatives by $\frac12$ by taking $r_2\geq2c_3c_1X^{1+2\delta}$. The second partial derivative of $f$ with respect to any $b_{ij}$ is a constant since $f$ is quadratic. Finally, the side lengths of $\B'$ are of the form $2X_{ij}/(rm)\gg X^{1-2\delta}/(r_1m)\gg 1$ by the assumption on $r_1$. Hence, we have 
\begin{eqnarray*}
\sum_{B' \in \B' \cap V(\Z)} e\left(\frac{\theta}{p} q(rmB'+B_1)\right) &=& \int_{\B'} e\left(\frac{\theta}{m} q(rmB'+B_1)\right) dB' + O\left(\left( \frac{X^{1 + 2\delta}}{rm}\right)^8\right)\\
&=&\frac{1}{r^9 m^9}\int_{\B} e\left(\frac{\theta}{m} q(B)\right) dB + O\left(\left( \frac{X^{1 + 2\delta}}{rm}\right)^8\right). 
\end{eqnarray*}
Summing over the $r^9$ possible $B_1$'s then gives
\begin{equation}\label{eq:onepoint}
S_\B(\alpha; m, B_0) = c_q(a; r, m, B_0) \int_{\B} e\left(\frac{\theta}{m} q(B)\right) dB + O\left(\frac{rX^{8(1 + 2\delta)}}{m^8}\right), 
\end{equation}
where 
$$ c_q(a; r, m, B_0) = \frac{1}{r^9 m^9} \sum_{\substack{B_1 \bmod{rm} \\ B_1 \equiv B_0 \bmod{m}}} e\left(\frac{a}{rm} q(B_1)\right). $$

In the light of \eqref{eq:Cq}, we define for any integer $r\geq1$ and any integer $a$ coprime to $r$, \[ c_q(a; r) = \frac{1}{r^9} \sum_{B \bmod{r}} e\left(\frac{a}{r} q(B)\right). \]

\begin{lemma}\label{lem:cqar}
If $\gcd(r, m) = 1$, then $\displaystyle c_q(a; r, m, B_0) = \frac{1}{m^9} c_q(a; r)$.
Otherwise $c_q(a; r, m, B_0) = 0$.
\end{lemma}
\begin{proof}
We consider the case $\gcd(r, m) = 1$ first.
Let $\bar{m}$ be any integer such that $m \bar{m} \equiv 1 \pmod{r}$.
For any integer $n$ divisible by $m$, we have $\frac{a}{rm} n \equiv \frac{a \bar{m}}{r} n\pmod{1}$.
Suppose now $B_1, B'\in V(\Z)$ with $B_1 \equiv B_0 \pmod{m}$ and $B_1 \equiv B' \pmod{r}$.
Then $q(B_1) \equiv q(B_0) \equiv 0 \pmod{m}$ and $q(B_1) \equiv q(B')\pmod{r}$ and so
\[ e\left(\frac{a}{rm} q(B_1)\right) = e\left(\frac{a\bar{m}}{r}q(B_1)\right) = e\left(\frac{a\bar{m}}{r} q(B')\right). \]
Since $m$ and $r$ are coprime, we have by the Chinese Remainder Theorem,
\[ \sum_{\substack{B_1 \bmod{rm} \\ B_1 \equiv B_0 \bmod{m}}} e\left(\frac{a}{rm} q(B_1)\right) = \sum_{B' \bmod{r}} e\left(\frac{a}{r} q(B')\right). \]
Dividing by $r^9 m^9$ gives us $\displaystyle c_q(a; r, m, B_0) = \frac{1}{m^9} c_q(a; r)$.

Now, we consider the case $\gcd(r, m) > 1$.
Suppose that $p$ is a prime dividing $\gcd(r, m)$.
We rewrite the sum as
\[ \sum_{\substack{B_1 \bmod{rm} \\ B_1 \equiv B_0 \bmod{m}}} e\left(\frac{a}{rm} q(B_1)\right) = \sum_{\substack{B' \bmod{rm/p} \\ B' \equiv B_0 \bmod{m}}} \sum_{B_0' \bmod{p}} e\left(\frac{a}{rm} q\left(\frac{rm}{p} B_0' + B'\right)\right). \]
Given $v, w \in V$, we write $\langle v, w \rangle = q(v + w) - q(v) - q(w)$ for the associated bilinear form.
Hence, we have
\[ q\left(\frac{rm}{p} B_0' + B'\right) = q(B') + \frac{rm}{p} \langle B', B_0' \rangle + \frac{r^2 m^2}{p^2} q(B_0'). \]
Since $rm \mid \frac{r^2 m^2}{p^2}$, the inner sum equals
\[ \sum_{B_0' \bmod{p}} e\left(\frac{a}{rm} \left(q(B') + \frac{rm}{p} \langle B', B_0' \rangle \right)\right) = e\left(\frac{a}{rm} q(B')\right) \sum_{B_0' \bmod{p}} e\left(\frac{a}{p} \langle B', B_0' \rangle \right). \]
Since $B' \equiv B_0$ is nonzero modulo $p$ and $q$ is non-degenerate modulo $p$ for $p \geq 3$, the linear form $\langle B', * \rangle : V(\F_p) \to \F_p$ is nonzero.
Moreover, $a$ is coprime to $p$ since $p \mid r$ and $a$ is coprime to $r$.
Therefore, the above exponential sum vanishes and as a result, $c_q(a; r, m, B_0) = 0$.
\end{proof}

Integrating \eqref{eq:onepoint} over the arc $|\theta|\leq \frac{1}{rr_2}$ and summing over $a$ and $r$ now give
\begin{eqnarray}
\nonumber		\int_{\kM} S(\alpha; m, B_0) d\alpha &=& \sum_{\substack{r \leq r_1 \\ \gcd(r, m) = 1}} \sum_{\substack{0 \leq a < r \\ \gcd(a, r) = 1}} \frac{1}{m^9} c_q(a; r) \int_{|\theta| \leq \frac{1}{rr_2}} \int_{\B} e\left(\frac{\theta}{m} q(B)\right) dB \; d\theta + O\left(\frac{r_1^2 X^{8(1 + 2\delta)}}{r_2 m^8}\right)\\
	&=& \sum_{\substack{r \leq r_1 \\ \gcd(r, m) = 1}} \frac{1}{m^8} C_q(r) \int_{|\theta| \leq \frac{1}{mrr_2}} \int_{\B} e(\theta q(B))\, dB \; d\theta + O\left(\frac{r_1^2 X^{8(1 + 2 \delta)}}{r_2 m^8}\right),\label{eq:major-arc1}
\end{eqnarray}

Our aim is to replace the above truncated sum by the singular series
\begin{equation}\label{eq:Smq}
    \kS_m(q) = \sum_{\gcd(r, m) = 1} \frac{1}{m^8} C_q(r)
\end{equation}
and the above integral by the singular integral $\kS_\infty(\B;q)$.
To this end, we prove the following bounds:

\begin{lemma}\label{lem:singular bound}
With notations as above, we have:
\begin{enumerate}[(a)]
    \item for all $r\geq 1$,
    \begin{equation}\label{eq:Cqbound}
    |C_q(r)| \leq 4 r^{-7/2};
    \end{equation}
    \item for all $\theta \neq 0$, 
    \begin{equation}\label{eq:Vbound}
    \int_\B e(\theta q(B))\,dB \ll \min\{X^9, |\theta|^{-9/2}\}.
    \end{equation}
    \item the singular integral 
    \begin{equation}\label{eq:Vtotalbound}
    \kS_\infty(\B;q)=\int_\R\int_\B e(\theta q(B))\,dBd\theta \ll X^7.
    \end{equation}
\end{enumerate}
The above implied constants depend only on $q$ (which is fixed).
\end{lemma}

\begin{proof}
We prove first the bound
\begin{equation}\label{eq:expquad}
    \sum_{B \bmod{r}} e\left(\frac{a}{r} q(B)\right) \leq 8 r^{9/2}.
\end{equation}
Also, we prove the bound with the constant $8$ replaced by $\sqrt{2}$ for $r$ odd.
Recall that $q(b_{ij})=-b_{11}b_{44} - b_{22}b_{33} - 2b_{12}b_{34} - 2b_{13}b_{24} - 2b_{14}^2$.
Hence \eqref{eq:expquad} follows from
\[ \sum_{x, y \bmod{r}} e\left(\frac{a}{r} xy\right) = \gcd(a, r) r, \quad \left|\sum_{x \bmod{r}} e\left(\frac{a}{r} x^2\right)\right| \leq (2 \gcd(a, r) r)^{1/2}, \]
    where $\gcd(a,r) \mid 2$.
Note that the second sum is a standard quadratic Gauss sum and the bound follows, for example, from \cite[\S1.3--\S1.6]{BEW}.
Thus, for $r$ odd, $|C_q(r)| \leq \sqrt{2} r^{-9/2} \phi(r) \leq \sqrt{2} r^{-7/2}$, where $\phi(r)$ is the Euler's totient function.
Meanwhile, for $r$ even, we have $|C_q(r)| \leq 8 r^{-9/2} \phi(r) \leq 4 r^{-7/2}$. This proves \eqref{eq:Cqbound}.

Next we prove the bound \eqref{eq:Vbound}. The $X^9$ bound is trivial since $\Vol(\B) \ll X^9$. We may also assume that $\theta > 0$ as the case $\theta < 0$ follows by complex conjugation. Setting $B'=\theta^{1/2}B$, we see that it suffices to prove the following general statement: for any box $\B'$ centered at the origin, \[ \int_{\B'} e(q(B')) \,dB' \ll 1. \]
Again, using the explicit formula of $q$, it reduces to proving that for any $X, Y > 0$, \[ \int_{-X}^X \int_{-Y}^Y e(xy) \,dydx\ll 1,\quad \int_{-X}^X e(x^2) dx \ll 1. \]
The first integral can be computed as follows:
    \[ \int_{-X}^X \int_{-Y}^Y e(xy) dx \; dy = \int_{-X}^X \frac{\sin(2 \pi xY)}{\pi x} dx = \int_{-XY}^{XY} \frac{\sin(2 \pi x)}{\pi x} dx \ll 1. \]
Now, we bound the second integral.
Since $e(x^2)$ is an even function, we can write \[ \int_{-X}^X e(x^2) dx = 2 \int_0^X e(x^2) dx. \]
For $0 < X < 1$, we can use the trivial estimate.
For $X \geq 1$, we use the trivial estimate for $x \in [0, 1]$ and partial integration for $x \in [1, X]$:
    \[ \int_0^X e(x^2) dx \ll 1 + \left[\frac{e(x^2)}{2x}\right]_1^X + \int_1^X \frac{e(x^2)}{2x^2} dx \ll 1 + 1 + \left[\frac{1}{2x}\right]_1^X \ll 1. \]
    
Finally, by using \eqref{eq:Vbound}, we have $$\kS_\infty(\B;q) \ll \int_{|\theta| \leq X^{-2}} X^9 d\theta + \int_{|\theta| \geq X^{-2}} |\theta|^{-9/2} d\theta \ll X^7,$$
which is the desired bound \eqref{eq:Vtotalbound}.
\end{proof}

Note the bound \eqref{eq:Cqbound} on $C_q(r)$ implies that 
$$|\kS(q) - 1| \leq 4(\zeta(7/2)-1) < 1$$
and that 
$$|\kS(q;p)-1| \leq 4\sum_{\ell\geq 1}p^{-(7/2)\ell} = \frac{4}{p^{7/2}-1} < 1.$$
Hence, the series defined by \eqref{eq:Sq} and \eqref{eq:Sqp} have positive values.

\vspace{5pt}
Combining the bounds \eqref{eq:Cqbound}, \eqref{eq:Vbound} and \eqref{eq:Vtotalbound} with \eqref{eq:major-arc1}, we have
\begin{eqnarray}
\nonumber	\int_{\kM} S(\alpha; p, B_0) d\alpha &=& \sum_{\substack{r \leq r_1 \\ \gcd(r, m) = 1}} \left(\frac{1}{m^8} C_q(r) (\kS_\infty(\B;q) + O((prr_2)^{7/2})\right) +O\left(\frac{r_1^2 X^{8(1 + 2 \delta)}}{r_2 m^8}\right)\\
	&=&\nonumber\left(\kS_m(q) + \sum_{r > r_1}O(r^{-7/2} m^{-8})\right)\kS_\infty(\B;q) + \sum_{r\leq r_1} O(r^{-7/2} m^{-8} (mrr_2)^{7/2}) +O\left(\frac{r_1^2 X^{8(1 + 2 \delta)}}{r_2 m^8}\right)\\
	&=&\kS_m(q) \kS_\infty(\B; q) + O\left(\frac{X^7}{r_1^{5/2} m^8} + \frac{r_1 r_2^{7/2}}{m^{9/2}} + \frac{r_1^2 X^{8(1 + 2 \delta)}}{r_2 m^8}\right), \label{eq:majorarc}
\end{eqnarray}
    where $\kS_m(q)$ is defined in \eqref{eq:Smq}.

\subsubsection{Minor arc estimate}

We now estimate the minor arc integral. Fix some $\alpha = \frac{a}{r} + \theta \in \km$, where $r_1 < r \leq r_2$ and $|\theta| \leq \frac{1}{r r_2}$. Then
\begin{eqnarray*} 
|S_\B(\alpha; m, B_0)|^2 &=& \sum_{\substack{B', B'' \in \B \cap V(\Z) \\ B', B'' \equiv B_0 \bmod{m}}} e\left(\frac{\alpha}{m} (q(B'') - q(B'))\right)\\
&=& \sum_{B\in V(\Z)}\sum_{\substack{B' \in \B \cap V(\Z) \\ B' \equiv B_0 \bmod{m} \\ B' + m B \in \B \cap V(\Z)}} e\left(\frac{\alpha}{m} (m^2q(B) + m \langle B', B\rangle)\right),
\end{eqnarray*}
where the second equality follows by setting $B'' = B' + mB.$ The set of $B\in V(\Z)$ for which the inner sum is non-empty is contained in the box $\B'' = \frac{1}{m}(\B-\B) = \{\frac{1}{m}(B''-B')\colon B',B''\in \B\}.$ Taking absolute values now give
\begin{equation}\label{eq:sumB}
|S_\B(\alpha; m, B_0)|^2 \leq \sum_{B\in \B''\cap V(\Z)} \left|\sum_{\substack{B' \in \B \cap V(\Z) \\ B' \equiv B_0 \bmod{m} \\ B' + m B \in \B \cap V(\Z)}} e\left(\alpha \langle B',B\rangle\right)\right| = \sum_{B\in \B''\cap V(\Z)} \left|\sum_{\substack{B_1 \in V(\Z) \\ B_0+mB_1\in\B \\ B_0 + mB_1+ m B \in \B}} e\left(\alpha m \langle B_1,B\rangle\right)\right|.
\end{equation}
Let $b_{ij}$ denote the entries of $B$ and let $x_{ij}$ denote the entries of $B_1$, then from the explicit formula for $q$, we have
\begin{equation}\label{eq:bilinear}
\langle B_1, B\rangle = -b_{11}x_{44}-b_{44}x_{11}-b_{22}x_{33}-b_{33}x_{22}-2b_{12}x_{34}-2b_{34}x_{12}-2b_{13}x_{24}-2b_{24}x_{13}-4b_{14}x_{14}.
\end{equation}
Each $x_{ij}$ takes all integer values within an interval, depending only on $b_{ij}$, of length at most $2X_{ij}/m$. Hence, the inner sum in \eqref{eq:sumB} factors into a product of geometric sums. For each $(i,j)$, let $\alpha_{ij}$ denote the integer coefficient in front of each $b_{ij}$ in \eqref{eq:bilinear} and let $I_{ij}$ denote the closed interval $[-2X_{ij}/m,2X_{ij}/m]$. Then we have
$$|S_\B(\alpha; m, B_0)|^2 \ll \prod_{(i,j)} \,\sum_{b_{ij}\in I_{ij}\cap \Z} \min\left\{\frac{X_{ij}}{m}, ||\alpha\alpha_{ij}mb_{ij}||^{-1}\right\},$$
where $||\cdot||$ is the distance to the nearest integer function.

Note we have $$|\theta \alpha_{ij} m b_{ij}| \leq \frac{4mb_{ij}}{rr_2} \leq \frac{8X_{ij}}{rr_2}\leq \frac{8c_1X^{1+2\delta}}{rr_2}\leq\frac{1}{2r}$$
by taking $r_2\geq 16c_1X^{1+2\delta}$. So we have the lower bound $$||\alpha\alpha_{ij}mb_{ij}||\geq\frac12||\frac{a}{r}m\alpha_{ij}b_{ij}||.$$ Write $r_{ij}=r/\gcd(r, m\alpha_{ij})\geq r/(4m)$ and $a_{ij}=am\alpha_{ij}/\gcd(r, m\alpha_{ij}).$ We have
$$|S_\B(\alpha; m, B_0)|^2 \ll \prod_{(i,j)} \sum_{b_{ij}\in I_{ij}\cap\Z} \min\left\{\frac{X_{ij}}{m},||\frac{a_{ij}}{r_{ij}}b_{ij}||^{-1}\right\}.$$
Now if $r_{ij}>4X_{ij}/m$, then 
\begin{equation}\label{eq:rbig}
\sum_{b_{ij}\in I_{ij}\cap\Z} \min\left\{\frac{X_{ij}}{m},||\frac{a_{ij}}{r_{ij}}b_{ij}||^{-1}\right\}\leq \frac{X_{ij}}{m} + 2\sum_{\ell=1}^{\lceil 2X_{ij}/m\rceil} \frac{r_{ij}}{\ell} \ll \frac{X_{ij}}{m} + r_{ij}\log X_{ij}.
\end{equation}
If $r_{ij}\leq 4X_{ij}/m$, then
\begin{equation}\label{eq:rsmall}
\sum_{b_{ij}\in I_{ij}\cap\Z} \min\left\{\frac{X_{ij}}{m},||\frac{a_{ij}}{r_{ij}}b_{ij}||^{-1}\right\}\ll \frac{X_{ij}}{m}\frac{4X_{ij}/m}{r_{ij}} + \frac{4X_{ij}/m}{r_{ij}}\sum_{\ell=1}^{\lfloor r_{ij}/2\rfloor} \frac{r_{ij}}{\ell} \ll \frac{X_{ij}^2}{rm} + \frac{X_{ij}}{m}\log X_{ij}.
\end{equation}
Combining \eqref{eq:rbig} and \eqref{eq:rsmall} then gives
$$\sum_{b_{ij}\in I_{ij}\cap\Z} \min\left\{\frac{X_{ij}}{m},||\frac{a_{ij}}{r_{ij}}b_{ij}||^{-1}\right\}\ll \frac{X^{2(1+2\delta)}}{rm} + r_2\,\log X.$$
Raising it to the power $9$ and taking square root give
\[ S_\B(\alpha; m, B_0) \ll \frac{X^{9(1 + 2 \delta)}}{r^{9/2} m^{9/2}} + r_2^{9/2} \, \log^{9/2} X. \]
Finally, integrating over the minor arc gives
\begin{eqnarray}
\nonumber\int_{\km} S_\B(\alpha; m, B_0)
    &\ll& \nonumber\sum_{r_1 < r \leq r_2} \sum_{\substack{0 \leq a < r \\ (a, r) = 1}} \int_{|\theta| \leq \frac{1}{r r_2}} \left(\frac{X^{9(1 + 2 \delta)}}{r^{9/2} m^{9/2}} + r_2^{9/2} \, \log^{9/2} X\right) d\theta \\
    &\ll& \nonumber\sum_{r_1 < r \leq r_2}\left(\frac{X^{9(1 + 2 \delta)}}{r_2r^{9/2} m^{9/2}}+r_2^{7/2}\log^{9/2} X\right)\\
    &\ll& \frac{X^{9(1 + 2 \delta)}}{r_2r_1^{7/2} m^{9/2}} + r_2^{9/2}\log^{9/2} X,\label{eq:minorIntegral}
\end{eqnarray}
where the last bound follows from $r_2\gg X^{1+2\delta}.$

\subsubsection{Proof of Theorem 4.2}

We are ready to prove Theorem \ref{thm:circlep}.
By \eqref{eq:majorarc} and \eqref{eq:minorIntegral}, we have \[ N_q(\B;m,B_0) = \kS_m(q) \kS_\infty(\B; q) + O\left(\frac{X^7}{r_1^{5/2} m^8} + \frac{r_1 r_2^{7/2}}{m^{9/2}} + \frac{r_1^2 X^{8(1 + 2 \delta)}}{r_2 m^8} + \frac{X^{9(1 + 2 \delta)}}{r_2r_1^{7/2} m^{9/2}} + r_2^{9/2} \log^{9/2} X\right), \]
    where $\kS_m(q)$ is defined in \eqref{eq:Smq}.
Since $X\gg m^3$ and $\delta < 0.01$, we may pick
$$r_1 = X^{\frac{2}{11}(1 + 2 \delta)} m^{\frac{7}{11}},\qquad r_2 = \frac{X^{1.522 (1 + 2 \delta)}}{m^{1.223} \log X}$$ so that $r_1m\ll X^{1-2\delta}$ and $X^{1+2\delta}\ll r_2\ll X^2$. With this choice of $r_1$ and $r_2$, we have \[ N_q(\B;m,B_0) = \kS_m(q) \kS_\infty(\B; q) + O\left(\frac{X^{6.85 (1 + 2 \delta)}}{m^{5.5}} \log X\right). \]
Finally, since $C_q$ is multiplicative (as easily verified), the singular series $\kS_m(q)$ defined in \eqref{eq:Smq} equals
\[ \kS_m(q) = \frac{1}{m^8} \left(\prod_{p \mid m} \kS(q; p)^{-1}\right) \kS(q). \]
This completes the proof of Theorem \ref{thm:circlep}.

\subsection{Proof of Theorem 4.1 using the Selberg sieve}

For any prime $p$, we say an element $B\in V(\F_p)$ (or $W(\F_p)$) is $\F_p$-\emph{reducible} if either $\Delta(B) = 0\in\F_p$ or $\Delta(B)\neq 0$ and $B$ is $\F_p$-distinguished in the sense of Section \ref{sec:invar}. Then any element $B\in V(\Z)$ that is $\Q$-distinguished is $\F_p$-reducible for all primes $p$.

We now apply the Selberg sieve to prove Theorem \ref{thm:circlemain}.
We follow the setup as in \cite[\S3]{ST}.
Let $z$ be a number less than $X^{1/3}$.
Let $P$ be the product of all primes $p$ with $N \leq p < z$ where $N$ is some large enough absolute constant.
For each $m \mid P$, let $a_m$ be the number of elements $B \in \B \cap V(\Z)$ such that:
\begin{itemize}
    \item $q(B) = 0$;
    \item for any prime $p \mid \frac{P}{m}$, $B$ is $\F_p$-reducible;
    \item for any prime $p \mid m$, $B$ is not $\F_p$-reducible.
\end{itemize}
For $m \nmid P$, we set $a_m = 0$.
Then, applying the Selberg sieve will give us the count for
\[ a_1 = \sum_{\gcd(n, P) = 1} a_n, \]
    which is the number of elements $B \in \B \cap V(\Z)$ with $q(B)=0$ and is $\F_p$-reducible for all primes $p \mid P$.

For any squarefree $m \mid P$, the expression \[ \sum_{n \equiv 0 \bmod{m}} a_n \] counts the number of elements $B \in \B \cap V(\Z)$ such that $q(B) = 0$ and $B$ is not $\F_p$-reducible for any $p \mid m$. Recall that for any prime $p$, we defined $d_p$ in Proposition \ref{prop:F_p-dist} for the number of $B_0\in W(\F_p)$ with $f_{B_0}\in U(\F_p)$ and are not $\F_p$-distinguished, which is the same as the number of $B_0\in V(\F_p)$ with $q(B_0)=0$ and are not $\F_p$-reducible. The condition that $\Delta(B_0)\neq 0$ in $\F_p$ also implies that $B_0$ is nonzero modulo $p$.
Thus, by Proposition \ref{prop:F_p-dist} and Theorem \ref{thm:circlep}, we have
\begin{eqnarray*}
    \sum_{n \equiv 0 \bmod{m}} a_n
    &=& \frac{1}{m^8} \prod_{p \mid m} \Big(d_p \kS(q; p)^{-1}\Big) \kS(q) \kS_\infty(\B; q) + O\left(X^{6.85(1 + 2\delta)} m^{2.5} \log X\right) \\
    &=& \left(\prod_{p \mid m} \frac{d_p \kS(q; p)^{-1}}{p^8}\right) \kS(q) \kS_\infty(\B; q) + O\left(X^{6.85(1 + 2\delta)} m^{2.5} \log X\right).
\end{eqnarray*}
We set $\displaystyle g(m) = \prod_{p \mid m} g(p)$ and $u_m = O\left(X^{6.85(1 + 2\delta)} m^{2.5} \log X\right)$ for each squarefree $m \mid P$ where 
\vspace{-15pt}$$g(p) = \frac{d_p \kS(q; p)^{-1}}{p^8}$$ for each prime $p \mid P$.
By \eqref{eq:Cqbound}, we have
\[ \kS(q; p) = 1+O\Big( \sum_{\ell \geq 1} p^{-7 \ell/2}\Big) = 1 + O(p^{-7/2}).\]
Recall from Proposition \ref{prop:F_p-dist}, we have the bound $$\frac{1}{16} + O(p^{-1})\, \leq\, \frac{d_p}{p^8}\, \leq\, \frac{3}{4} + O(p^{-1}).$$ Hence, by taking $N$ large enough, we have the bound $\displaystyle\frac{1}{32} \leq g(p) \leq \frac{7}{8}$ for $p\geq N$.

Now, set $\displaystyle h(m) = \prod_{p \mid m} \frac{g(p)}{1 - g(p)}$ for all squarefree $m \mid P$.
Let $D > 1$ with $D<z$ be a real number to be picked later and set
\[ H = \sum_{\substack{m < \sqrt{D} \\ m \mid P}} h(m). \]
Then, by \cite[Theorem 6.4]{IK}, we have
\[ a_1 = \sum_{\gcd(n, P) = 1} a_n \leq H^{-1} \kS(q) \kS_\infty(\B; q) + R, \]
    where
\[ |R| \leq \sum_{\substack{m < \sqrt{D} \\ m \mid P}} \tau_3(m) u_m \ll_\epsilon X^{6.85(1 + 2\delta)} \log X \sum_{m < \sqrt{D}} m^{2.5 + \epsilon} \ll_\epsilon X^{6.85(1 + 2\delta)} D^{1.75 + \epsilon} \log X \]
    for any $\epsilon > 0$.
    
Meanwhile, for $p$ prime, we have $\frac{1}{31} \leq h(p) \leq 8$ and so for any $\epsilon > 0$,
\[ H \gg \pi(\sqrt{D}) \gg_\epsilon D^{0.5 - \epsilon}. \]
Thus, we get
\[ N_q^\dist(\B) \leq a_1 \ll_\epsilon X^7 D^{-0.5 + \epsilon} + X^{6.85(1 + 2\delta)} D^{1.75 + \epsilon} \log X. \]
Taking $D = X^{(1/15) - (54.8/9)\delta}$ gives the desired bound \eqref{eq:circlemain}.

%% file: proof.tex
\section{Proof of Theorem \ref{thm:main} and Theorem \ref{thm:tailestimate}}\label{sec:proof}

We prove Theorem \ref{thm:tailestimate} first. By the paragraph following Theorem \ref{thm:tailestimate}, it is enough to consider the case $\alpha = 256$ and $\beta=-27$. We note that for $(a,b)\in\Z^2$ with $H(a,b)<X$, there are at most $X^\epsilon$ integers $m$ whose square divides $\Delta(a,b)$. Hence, it is enough to prove the bound \eqref{eq:final} for $$\#\bigcup_{\substack{m>M\\ m\;\mathrm{ squarefree}}} \{(a,b)\in\Z^2\colon H(a,b)<X, m^2\mid \Delta(a,b)\}.$$
Moreover, if $m^2\mid\Delta(a,b)$, then we can factor $m = m_1m_2$ where $m_1$ is the product of all prime factors $p$ of $m$ such that $p^2$ strongly divides $\Delta(a,b)$, and $m_2$ is the product of all prime factors $p$ of $m$ such that $p^2$ weakly divides $\Delta(a,b)$. Since at least one of $m_1$ or $m_2$ is at least $m'$ for some squarefree integer $m'\geq\sqrt{m}$, we have
$$\bigcup_{\substack{m>M\\ m\;\mathrm{ squarefree}}}\{(a,b)\in\Z^2\colon H(a,b)<X, m^2\mid \Delta(a,b)\}\quad\subset \bigcup_{\substack{m'>\sqrt{M}\\ m'\;\mathrm{ squarefree}}}\W_{m'}^{(1)}\quad \cup \bigcup_{\substack{m'>\sqrt{M}\\ m'\;\mathrm{ squarefree}}}\W_{m'}^{(2)}.$$
Theorem \ref{thm:tailestimate} now follows from Theorem \ref{thm:weak}.

\vspace{5pt}
Finally, we prove Theorem \ref{thm:main} using an inclusion-exclusion sieve. We have
$$N(X;\alpha,\beta)=\sum_m\mu(m)N_m(X;\alpha,\beta),$$
and the following individual count
$$
N_m(X;\alpha,\beta) = X^7m^{-4}\rho_{\alpha,\beta}(m^2) + O(X^4m^{-2}\rho_{\alpha,\beta}(m^2)) + O(\rho_{\alpha,\beta}(m^2)).
$$
Since $\rho_{\alpha,\beta}(m^2) = O(m^2)$, we sum over $m<X^\eta$ for some $\eta>0$ to get
\begin{eqnarray}
\nonumber\sum_{m\leq X^\eta}\mu(m)N_m(X;\alpha,\beta) &=& X^7\sum_{m\leq X^\eta} \frac{\rho_{\alpha,\beta}(m^2)}{m^4} + O(X^{4+\eta}) + O(X^{1+3\eta})\\
&=& C(\alpha,\beta)X^7 + O(X^{7-\eta}) + O(X^{4+\eta}) + O(X^{1+3\eta}).\label{eq:mainterm}
\end{eqnarray}
We take $\eta = 0.1$ and apply Theorem \ref{thm:tailestimate} with $M = X^{0.1}$ to get
$$N(X;\alpha,\beta) = C(\alpha,\beta)X^7 + O(X^{6.9}) + O_\epsilon(X^{6.9+\epsilon} + X^{6.992+\epsilon}).$$
The proof of Theorem \ref{thm:main} is now complete.